\documentclass[10pt]{amsart}
\usepackage{txfonts}
\usepackage{amscd}
\usepackage{amssymb}
\usepackage{amsmath}
\usepackage{amscd,mathrsfs,latexsym,amsmath,amsthm,amssymb,amsxtra,}
\usepackage{latexsym,amsmath,amsthm,amssymb,amsxtra,mathrsfs}
\usepackage[bookmarks,colorlinks]{hyperref}
\usepackage[all]{xy}
\RequirePackage{amsmath} \RequirePackage{amssymb}
\usepackage{amscd,latexsym,amsthm,amsfonts,amssymb,amsmath,amsxtra}
\usepackage{color}
\usepackage{bbm}
\usepackage{colordvi}
\usepackage{multicol}

    \theoremstyle{plain}

    \newtheorem{thm}{Theorem}[section] \newtheorem{cor}[thm]{Corollary}
    \newtheorem{lem}[thm]{Lemma}  \newtheorem{prop}[thm]{Proposition}

     \newtheorem {example}[thm]{Example}
    
    \numberwithin{equation}{section}

\begin{document}
\newcommand{\Li}{\mathscr{L}}
\newcommand{\Qp}{\mathbf{Q}_{p}}
\title[]{$\Li$-invariants and logarithm derivatives of eigenvalues of Frobenius}

\author{Zhang Yuancao}
\address{Beijing International Center for Mathematical Research, Beijing, 100871, China}
\email{zhangyuancao@pku.edu.cn}

\maketitle


\newcommand{\A}{\mathbf{A}}
\newcommand{\GK}{\mathscr{G}_{K}}
\newcommand{\Zp}{\mathbf{Z}_{p}}
\newcommand{\Hk}[2]{H^{#1}({#2})}
\newcommand{\LKz}{L\otimes_{\Qp}K_{0}}

\newcommand{\LK}{L\otimes_{\Qp}K}
\newcommand{\OK}{\mathscr{O}_{K}}

\newcommand{\X}{\mathscr{X}}
\newcommand{\Nm}{\mathrm{Nm}}
\newcommand{\hotimes}{\hat{\otimes}}
\newcommand{\BHT}{\mathbf{B}_{\mathrm{HT}}}
\newcommand{\BdR}{\mathbf{B}_{\mathrm{dR}}}
\newcommand{\Bcris}{\mathbf{B}_{\mathrm{cris}}}
\newcommand{\Bst}{\mathbf{B}_{\mathrm{st}}}
\newcommand{\MFad}{\mathbf{MF}^{\mathrm{ad}}_{K,L}(\varphi, N)}
\newcommand{\Xst}{\mathbf{X}_{\mathrm{st}}}
\newcommand{\Vst}{\mathbf{V}_{\mathrm{st}}}
\newcommand{\Dst}{\mathbf{D}_{\mathrm{st}}}
\newcommand{\Dcris}{\mathbf{D}_{\mathrm{cirs}}}
\newcommand{\XdR}{\mathbf{X}_{\mathrm{dR}}}
\newcommand{\Repst}{\mathbf{Rep}^{\mathrm{st}}_{L}(\GK)}
\begin{abstract}
Let $K$ be a $p$-adic local field. In this work we study a special kind of $p$-adic Galois representations of it. These representations are similar to the Galois representations occurred in the exceptional zero conjecture for modular forms. In particular, we verify that a formula of Colmez can be generalized to our case. We also include a degenerated version of Colmez's formula. 
\end{abstract}

\section{Introduction}

Let $f$ be a level $\Gamma_{0}(N)$ ($p \parallel N$) newform of even weight $k \geq 2$. Let $L(f,s)$ be the classical $L$-function associate to it. Denote the $p$-adic $L$-function associated to $f$ by $L_{p}(f,s)$. In \cite{MTT}, Mazur, Tate and Teitelbaum noted that $L_{p}(f,s)$ has one more trivial zero than $L(f,s)$ at the point $s=k/2$. Let $\rho_{f}$ be the Galois representation of $\mathscr{G}_{\mathbf{Q}_{p}}$ associated to $f$. They conjectured further the ratio $\frac{L_{p}'(f,k/2)}{L(f,k/2)}$ depends only on $\rho_{f}$. The conjecture is called the exceptional zero conjecture.

For $k=2$ case, Mazur, Tate and Teitelbaum gave a conjectural local description of this constant $\mathscr{L}$. The conjecture in this case was proved by Greenberg and Stevens \cite{GS} . A key ingredient of their proof is Hida's family of $p$-adic ordinary Hecke eigenforms. For any weight $2$ newform $f$, there is one such family containing $f$. Let  $\alpha$ be the function of $U_{p}$-eigenvalues of the forms on the family. By some cohomology computation of the deformation of $\rho_{f}$, they proved the following formula:
\begin{gather}\label{eqn:stevens}
\mathscr{L} = -2\frac{\alpha'(f)}{\alpha(f)}
\end{gather}
for the $\mathscr{L}$-invariant constructed by Mazur, Tate and Teitelbaum. On the other hand, they used the family to construct a two variable $p$-adic $L$-function. The two variable $p$-adic $L$-function  was used to prove that 
\begin{gather}\label{eqn:gs1}
\frac{L'_{p}(f,k/2)}{L(f,k/2)} = -2\frac{\alpha'(f)}{\alpha(f)}.
\end{gather} 

For a general $k$, there are several ways to construct the local invariant $\mathscr{L}$. One is provided by Fontaine using his theory of semi-stable representations (see \cite{M}). Replacing Hida's families by Coleman's family of $p$-adic Hecke eigenforms, the same strategy in the proof of the formula (\ref{eqn:gs1}) applies to the general weight case. So to prove the exceptional zero conjecture, it remains to generalize the formula (\ref{eqn:stevens}). This was accomplished in \cite{G} by Stevens for the $\mathscr{L}$-invariants constructed by Coleman.  Inspired by this, Colmez \cite{C} proved a formula for Fontaine's $\mathscr{L}$-invariant. Applying this formula to Coleman's families, he proved
that (\ref{eqn:stevens}) also holds for Fontaines' $\mathscr{L}$-invariant. The key tool of Colmez's proof is the fundamental exact sequence of $p$-adic Hodge theory. Using it, one can carry out cohomology computation on the level of $\mathbf{X}_{\mathrm{st}}$ like the weight $2$ case. In this paper, we generalize the aforementioned formula of Colmez to general $p$-adic local fields. 

Before stating the main result, we introduce some notation first.

Let $K$ be a fixed extension of $\mathbf{Q}_{p}$ of degree $n$. Denote the maximal unramified extension of $\mathbf{Q}_{p}$ contained in $K$ by $K_{0}$.  Let $e$ and $f$ be the ramification index and the inertia degree of $K$, respectively. Denote the absolute Galois group of $K$ by $\mathscr{G}_{K}$. Let $I_{K}$ be the inertia subgroup of $\mathscr{G}_{K}$.  Denote the abelianization of $\mathscr{G}_{K}$ by $\mathscr{G}_{K}^{\mathrm{ab}}$. Let $I_{K}^{\mathrm{ab}}$ be the image of $I_{K}$ in $\mathscr{G}_{K}^{\mathrm{ab}}$. Throughout this note, the cohomology groups mean the continuous group cohomology of $\mathscr{G}_{K}$ if there is no confusion.

Fix an algebraic closure $\bar{\mathbf{Q}}_{p}$ of $\mathbf{Q}_{p}$. Let $K^{\mathrm{Gal}}$ be the Galois closure of $K$ in $\bar{\mathbf{Q}}_{p}$. The algebra $K^{\mathrm{Gal}}\otimes_{\mathbf{Q}_{p}}K$ is an Artin algebra. Consider the map
\begin{eqnarray*}
K^{\mathrm{Gal}}\otimes_{\mathbf{Q}_{p}}K & \to &\prod_{\sigma}K^{\mathrm{Gal}} \\
a\otimes b  &\mapsto &(a\cdot\sigma(b))_{\sigma} 
\end{eqnarray*}
where $\sigma$ runs over $\mathbf{Q}_{p}$-algebra embeddings from  $K$ onto $\bar{\mathbf{Q}}_{p}$. It is an isomorphism. We identify these two $\mathbf{Q}_{p}$-algebras from now on.

Let $\mathscr{X}=\mathrm{Sp}(S)$ be a $K^{\mathrm{Gal}}$-affinoid rigid space. Then $S$ is the quotient algebra of some Tate algebra $\mathbf{Q}_{p}\langle X_{1},\cdots, X_{d} \rangle$. We can endow $S$ with the quotient norm. In this way, $S$ becomes a topological $\mathbf{Q}_{p}$-algebra. Let $L$ be a fixed finite extension of $\mathbf{Q}_{p}$. The trace map from $K$ to $\mathbf{Q}_{p}$ induces a trace map from $S\otimes_{\mathbf{Q}_{p}}K$ to $S$. We denote the latter also by $\mathrm{tr}$. Suppose that $L$ contains $K^{\mathrm{Gal}}$. Denote the set of $L$-rational points of $\mathscr{X}$ by $\mathscr{X}(L)$. Identify $L\otimes_{\mathbf{Q}_{p}}K$ with $\prod_{\sigma}L$ as above. Let $L(n)$ be the the vector space $L$ such that $\mathscr{G}_{K}$ acts on it by $n$-th power of the cyclotomic character.

Let $\mathbf{B}_{\mathrm{st}}$, $\mathbf{B}_{\mathrm{dR}}$ and $\mathbf{B}_{\mathrm{cris}}$ be the semi-stable, de Rham and crystalline Fontaine rings, respectively. Denote their tensor products with $L$ over $\mathbf{Q}_{p}$ by $\mathbf{B}_{\mathrm{st},L}$, $\mathbf{B}_{\mathrm{dR},L}$ and $\mathbf{B}_{\mathrm{cris},L}$, respectively. Let $\mathbf{B}_{e}$ be the ring of elements stabilized by $\varphi$ in $\mathbf{B}_{\mathrm{cris}}$. Denote $\mathbf{B}_{e} \otimes_{\mathbf{Q}_{p}}L$ by $\mathbf{B}_{e,L}$. Let $\mathbf{Q}_{p}^{ur}$ be the $p$-adic completion of the maximal unramified extension of $\mathbf{Q}_{p}$. This field is contained in $\mathbf{B}_{\mathrm{cris}}$.  

We fix once and for all a system of primitive $p^{m}$-th roots of unity $\epsilon_{m}$ such that $\epsilon_{m+1}^{p} = \epsilon_{m}$. Denote it by $\epsilon$. Let $t$ be the element $\log \epsilon$ in $\mathbf{B}_{\mathrm{cris}}$. The $\varphi$-action on $t$ is multiplication by $p$. The Galois action on $t$ is via the cyclotomic character. 

Let $u$ be the element $\log p$ in $\mathbf{B}_{\mathrm{st}}$. The operator $N$ sends $u$ to $-1$. The $\varphi$-action on $u$ is multiplication by $p$.

Let $\psi_{1}:\mathscr{G}_{K} \to L$ be the continuous homomorphism which is trivial on $I_{K}$ and takes the value $f$ at the Frobenius element of $\mathscr{G}_{K}$. Let $\psi_{2}:\mathscr{G}_{K} \to L$ be the logarithm of cyclotomic character. Using the local class field theory, we can identify the Lie algebra of $I_{K}^{\mathrm{ab}}$ with $K$.   

By a family of $d$ dimensional Galois representations over $\mathscr{X}$, we mean a free $S$-module $\mathbf{V}$ of rank $d$ with continuous $S$-linear $\mathscr{G}_{K}$-action. The family $\det \mathbf{V}$ is a free $S$-module $\mathbf{V}$ of rank $1$, or equivalently we have a continuous homomorphism $\det \mathbf{V}:\mathscr{G}_{K} \to S^{\times}$ . Then taking logarithm we get a continuous homomorphism form $\mathscr{G}_{K}$ to $S$. The homomorphism is determined by the associated homomorphism of Lie algebras on $I_{K}^{\mathrm{ab}}$ and its value at a  fixed Frobenius element. Note that the space of homomorphisms of Lie algebra of $I_{K}^{\mathrm{ab}}$ to $S$ is $\mathrm{Hom}_{\mathbf{Q}_{p}}(K,S)$ via local class theory. In this way it can be viewed as a free $K\otimes_{\mathbf{Q}_{p}}S$-module of rank $1$. Moreover for any $\kappa \in K\otimes_{\mathbf{Q}_{p}}S$  there is a unique homomorphism $\kappa \psi_{2}$ from $\mathscr{G}_{K}$ to $S$ such that the corresponding map on Lie algebra of $I_{K}^{\mathrm{ab}}$ is $\kappa\cdot d\psi_{2}$ and the value at the element in $\mathscr{G}_{K}$ correspoding to $p$ under Artin map is $0$. Then any homomorphism between $\mathscr{G}_{K}$ and $S$ must be of the form $\delta \psi_{1}+ \kappa \psi_{2}$ where $\delta \in S$ and $\kappa \in K\otimes_{\mathbf{Q}_{p}}S$. 

Our main result is the following (see section 4 for the definition of a Galois representation of monodromy type).
\begin{thm}\label{thm:main}
Let $\mathbf{V}$ be a family of two dimensional Galois representations over $\mathscr{X}$. Let $\log \det \mathbf{V} = \delta \psi_{1}+ \kappa \psi_{2}$  where $\delta \in S$ and $\kappa \in K\otimes_{\mathbf{Q}_{p}}S$. Suppose that $((\mathbf{V}\otimes_{\mathbf{Q}_{p}}\mathbf{B}_{\mathrm{cris}})^{\varphi^{f}=\alpha})^{\mathscr{G}_{K}}$ is a free $S\otimes_{\mathbf{Q}_{p}}K_{0}$-module of rank $1$ for some $\alpha \in S$. Let $x$ be a point in $\mathscr{X}(L)$ such that $V_{x}$ is of monodromy type with Fontaine's $\mathscr{L}$-invariant $\mathscr{L} \in L\otimes_{\mathbf{Q}_{p}}K$. Then the differential form
$$\frac{d\alpha}{f\alpha}+\frac{1}{2}d\delta-\frac{1}{2n}\mathrm{tr} (\mathscr{L}\cdot d\kappa)$$
vanishes at the point $x$.
\end{thm}
Note if $K=\mathbf{Q}_{p}$, this recovers Colmez's result. We also have a ``degenerated'' version of this formula (see Proposition \ref{prop:degen}).

The plan of this paper is as follows. In the second and the third section, we recall some facts of local class field theory and $p$-adic Galois representation theory, respectively. We give the structure of $\mathbf{D}_{\mathrm{st}}(\mathbf{V}_{x})$ in the fourth section. Then we study the representation $\mathrm{End}^{0}(\mathbf{V}_{x})$ in detail in section 5. In particular we establish a cohomological property  \ref{prop:cohom} about $\mathscr{L}$-invariants. In section 6, we investigate a special kind of extensions of  the trivial $(\varphi,N)$-module. They are closely related to the constraint condition on $((\mathbf{V}\otimes_{\mathbf{Q}_{p}}\mathbf{B}_{\mathrm{cris}})^{\varphi^{f}=\alpha})^{\mathscr{G}_{K}}$.  We combine the results in the previous sections to prove our main theorem in the last section.

\noindent{\bf Acknowledgement}.
The author thanks his advisor Qingchun Tian. Without his suggestion and encouragement, the current work could not be accomplished. The author thanks Pierre Colmez for his enlightening lectures in Morningside Center of Mathematic on his works on $p$-adic local Langlands and for his suggestion to consider the ``degenerated'' cases and informing me Schraen's work in \cite{S}. The author thanks Tong Liu, Bingyong Xie, Liang Xiao and Yi Ouyang for their lectures introducing the $p$-adic Galois representation theory. The author thanks Ye Tian and Song Wang for organizing the workshop which provided a great opportunity to learn advanced mathematics. The author thanks Ruochuan Liu and Liang Xiao for much useful discussion.

\section{Some facts of local class field theory} ~\vskip 1pt
In this section we recall some facts of local class field theory. All the $\mathscr{G}_{K}$-representations are $L$-linear. All Galois cohomology group are for K.

The theory of Kummer gives us a canonical isomorphism  
\begin{eqnarray*}
L\otimes_{\mathbf{Z}_{p}}(\displaystyle{\varprojlim_{m}}K^{\times}/(K^{\times})^{p^{m}})& \to&\Hk{1}{L\otimes_{\mathbf{Q}_{p}}\mathbf{Q}_{p}(1)} \\
\sum_{i}a_{i}\otimes \alpha_{i}&\mapsto &\sum_{i}a_{i}\otimes(\alpha_{i})
\end{eqnarray*}
where the map $(\quad): \mathscr{G}_{K} \to \mathbf{Q}_{p}(1)$ is defined by 
$$\frac{g(\sqrt[p^{m}]{\alpha})}{\sqrt [p^{m}]{\alpha}} = \epsilon_{m}^{(\alpha)(g)} $$ 
for $\alpha  \in K^{\times}$ and $g \in \mathscr{G}_{K}$  ($(\sqrt[p^{m+1}]{\alpha})^{p} = \sqrt[p^{m}]{\alpha}$). 

On the other hand, the exponential map $$\exp: p\mathscr{O}_{K} \to K^{\times}$$  
extends to an embedding of $L\otimes_{\mathbf{Q}_{p}}K$ into  $\Hk{1}{L(1)}$ and induces a decomposition $\Hk{1}{L(1)} = \exp(L\otimes_{\mathbf{Q}_{p}}K)\oplus (p) L$. We identify $L\otimes_{\mathbf{Q}_{p}}K$ with its image $\exp(L\otimes_{\mathbf{Q}_{p}}K)$ in $\Hk{1}{L(1)}$ from now on.

In the cohomology group $\Hk{1}{L}$, there is an element $\psi_{1}$ which is trivial on the inertial subgroup of $\mathscr{G}_{K}$ and takes the value $f$ at the Frobenius element of $\mathscr{G}_{K}$ and another element $\psi_{2}$ which is the logarithm of the cyclotomic character of $\mathscr{G}_{K}$.

The local class field theory tells us $\Hk{2}{L(1)}$ is isomorphic to $L$ and the pairing $$\cup:\Hk{1}{L} \times \Hk{1}{L(1)} \to \Hk{2}{L(1)}$$ defined by the cup product is perfect. It induces a surjective homomorphism  from $\Hk{1}{L}$ to $\mathrm{Hom}_{L}(L\otimes_{\mathbf{Q}_{p}}K, \Hk{2}{L(1)})$.  The kernel is spaned by $\psi_{1}$. The decomposition 
$$\Hk{1}{L(1)} = \exp(L\otimes_{\mathbf{Q}_{p}}K)\oplus (p) L$$ induces a decompostion $\Hk{1}{L} = \mathrm{Hom}_{L}(L\otimes_{\mathbf{Q}_{p}}K, \Hk{2}{L(1)}) \oplus \psi_{1}L$. Note that $\psi_{2}$ is a generator of the free $L\otimes_{\mathbf{Q}_{p}}K$-module $\mathrm{Hom}_{L}(L\otimes_{\mathbf{Q}_{p}}K, \Hk{2}{L(1)})$. 

\begin{prop}\label{prop:calcupprod}
Let $a_{1} \in L$ and $a_{2} \in L\otimes_{\mathbf{Q}_{p}}K$. The cup product of $a_{1}\psi_{1} + a_{2}\psi_{2}$ and $b_{1}(p) + (\exp(b_{2}))$ is $$(a_{1}b_{1} - \frac{1}{n}\mathrm{tr}(a_{2}b_{2}))((\psi_{1}\cup(p)).$$
\end{prop}
\begin{proof}
The case $K = \mathbf{Q}_{p}$ is well known. For general $K$, consider the restriction map $\mathrm{res}$ and the corestriction map $\mathrm{cor}$ between $\mathscr{G}_{K}$ and $\mathscr{G}_{\mathbf{Q}_{p}}$. Then 
\begin{align*}
\mathrm{cor}((a_{1}\psi_{1} + a_{2}\psi_{2}) \cup (b_{1}(p) + (\exp(b_{2})))) &=a_{1}b_{1}(\psi_{1}\cup \mathrm{cor}((p)))+\psi_{2} \cup \mathrm{cor}((\exp(a_{2}b_{2})))\\
&=na_{1}b_{1}(\psi_{1}\cup (p)) +\mathrm{tr}(a_{2}b_{2}) \psi_{2} \cup (\exp(1))\\
&=(na_{1}b_{1}-\mathrm{tr}(a_{2}b_{2}))(\psi_{1}\cup ((p)))\\
&=(a_{1}b_{1} - \frac{1}{n}\mathrm{tr}(a_{2}b_{2}))\mathrm{cor}((\psi_{1}\cup (p)).
\end{align*}
The map $\mathrm{cor}$ is an isomorphism so the equality in the proposition holds.
\end{proof}


\section{Admissible $(\varphi,N)$-filtered modules and semi-stable representations} ~\vskip 1pt
A $(\varphi,N)$-module $D$ over $K$ with coefficient in $L$  is a free $L\otimes_{\mathbf{Q}_{p}}K_{0}$-module of finite rank with $(\varphi,N)$-action such that:
\begin{itemize}
\item The $\varphi$-action is semi-linear and bijective.
\item The $N$-action is linear.
\item $N\varphi=p\varphi N$.
\end{itemize}
We omit the coefficient field $L$ and the base field $K$ from now on.

Let $D$ be a $(\varphi,N)$-module. Denote its dimension as an $L$-vector space by $d$. Suppose that $v$ is a generator $e$ of $\wedge^{d}D$ as an $L$-vector space. Suppose that $\varphi(v) =\alpha v$ for some $\alpha \in L^{\times}$.   Define the Newton number $t_{N}(D)$ of $D$ as $e\cdot v_{p}(\alpha)$. Define a functor from the category of $(\varphi,N)$-modules to the category of $\mathbf{B}_{e,L}$-modules with Galois actions by
$$\mathbf{X}_{\mathrm{st}}(D) :=(D\otimes_{K_{0}}\mathbf{B}_{\mathrm{st}})^{\varphi=1, N=0}.$$ 
The Galois action is induced by the action on $\mathbf{B}_{\mathrm{dR}}$. By Proposition 5.1 in \cite{CF}, we know the functor is exact. 

A filtered $(\varphi,N)$-module $(D, \mathrm{Fil})$ is a $(\varphi,N)$-module $D$ with a decreasing separated exhausted $L\otimes_{\mathbf{Q}_{p}}K$-filtration $\mathrm{Fil}$ on $D_{K}=D\otimes_{K_{0}}K$. Define its Hodge number by the following formula 
$$t_{H}((D_{K}, \mathrm{Fil})):=\sum_{i}i\cdot\dim_{L} (\mathrm{Fil}^{i}(D_{K})/\mathrm{Fil}^{i+1}(D_{K})).$$ 
The elements in the set 
$$\{i|\mathrm{Fil}^{i}(D_{K})/\mathrm{Fil}^{i+1}(D_{K}) \neq 0\}$$
are called the Hodge-Tate weights of this filtered module.

A filtered $(\varphi,N)$-submodule $(D', \mathrm{Fil})$ of $(D, \mathrm{Fil})$ is a $(\varphi,N)$-stable $L\otimes_{\mathbf{Q}_{p}}K_{0}$-submodule of $D$ with the induced filtration, i.e. the filtration defined by:
$$\mathrm{Fil}^{i}(D'_{K}) := \mathrm{Fil}^{i}(D_{K}) \cap (D'_{K}).$$

We define a functor $\mathbf{X}_{\mathrm{dR}}$ from the category of filtered $(\varphi,N)$-modules to the category of $\mathbf{B}_{\mathrm{dR},L}/\mathbf{B}^{+}_{\mathrm{dR},L}$-modules with Galois actions by
$$\mathbf{X}_{\mathrm{dR}}((D,\mathrm{Fil})):= (D_{K}\otimes_{K}\mathbf{B}_{\mathrm{dR}})/\mathrm{Fil}^{0}(D_{K}\otimes_{K}\mathbf{B}_{\mathrm{dR}}), $$
where the filtration on $D_{K}\otimes_{K}\mathbf{B}_{\mathrm{dR}}$ is given by
$$\mathrm{Fil}^{k}(D_{K}\otimes_{K}\mathbf{B}_{\mathrm{dR}}) :=\sum_{i+j=k}\mathrm{Fil}^{i}(D_{K})\otimes_{K}t^{j}\mathbf{B}_{\mathrm{dR}}^{+}.$$ 
The Galois action is induced by the action on $\mathbf{B}_{\mathrm{dR}}$.

A filtered $(\varphi,N)$-module $(D, \mathrm{Fil})$ is called admissible if it satisfies the following conditions:
\begin{itemize}
\item $t_{N}(D)=t_{H}((D_{K}, \mathrm{Fil}))$.
\item For any filtered $(\varphi,N)$-submodule $D'$ of $D$, $t_{N}(D') \geq t_{H}((D'_{K}, \mathrm{Fil})).$
\end{itemize}

\begin{example}
Let $\alpha$ be an element in $L^{\times}$,  $k=(k_{\sigma})_{\sigma}$   elements in $\displaystyle{\oplus_{\sigma}}\mathbf{Z}$ satisfying $v_{p}(\alpha) = \sum_{\sigma} k_{\sigma}$. Then the following filtered $(\varphi,N)$-module is admissible.
$$D := L\otimes_{\mathbf{Q}_{p}}K_{0}\cdot e$$
$$\begin{array}{cc}
\varphi(e) = (\alpha,1,\cdots,1) v   \\
N(v) = 0 
\end{array}$$
$$\mathrm{Fil}_{\{k_{\sigma}\}}^{k}(D_{K}) := \oplus_{k_{\sigma} \leq k} (D_{K})_{\sigma}.
$$
\end{example}
Denote the category of admissible $(\varphi,N)$-filtered modules by $\mathbf{MF}^{\mathrm{ad}}_{K,L}(\varphi, N)$. It is a Tannakian category.

For an object $(D,\mathrm{Fil})$ in $\mathbf{MF}^{\mathrm{ad}}_{K,L}(\varphi, N)$, define
$$\mathbf{V}_{\mathrm{st}}((D, \mathrm{Fil})):= \ker(\mathbf{X}_{\mathrm{st}}(D) \to \mathbf{X}_{\mathrm{dR}}((D,\mathrm{Fil}))).
$$
It is an $L$-vector space with continuous Galois action.

On the other hand, consider continuous $\mathscr{G}_{K}$-representations over finite dimension $L$-vector spaces. Such a representation $V$ is called semi-stable if 
$$\mathrm{rank}_{L\otimes_{\mathbf{Q}_{p}} K_{0}}(V\otimes_{\mathbf{Q}_{p}}\mathbf{B}_{\mathrm{st}})^{\mathscr{G}_{K}} = \dim_{L}V.$$
Denote the category of semi-stable continuous $\mathscr{G}_{K}$-representations over finite dimension $L$-vector spaces by $\mathbf{Rep}^{\mathrm{st}}_{L}(\mathscr{G}_{K})$. It is also a Tannakian category. For any object $V$ in $\mathbf{Rep}^{\mathrm{st}}_{L}(\mathscr{G}_{K})$, define
$$\mathbf{D}_{\mathrm{st}}(V) := (V\otimes_{\mathbf{Q}_{p}}\mathbf{B}_{\mathrm{st}})^{\mathscr{G}_{K}}.$$ 

We have the following theorem of Colmez and Fontaine:
\begin{thm}[Colmez-Fontaine]\cite{CF}\label{thm:strep}$\,$

\begin{enumerate}
\item The functor $\mathbf{D}_{\mathrm{st}}$(resp. $\mathbf{V}_{\mathrm{st}}$) is a well defined equivalence from $\mathbf{Rep}^{\mathrm{st}}_{L}(\mathscr{G}_{K})$ to $\mathbf{MF}^{\mathrm{ad}}_{K,L}(\varphi, N)$( resp. from $\mathbf{MF}^{\mathrm{ad}}_{K,L}(\varphi, N)$ to $\mathbf{Rep}^{\mathrm{st}}_{L}(\mathscr{G}_{K})$).
\item For any object $(D, \mathrm{Fil})$ in $\mathbf{MF}^{\mathrm{ad}}_{K,L}(\varphi, N)$, the following sequence is exact:
$$0 \to \mathbf{V}_{\mathrm{st}}(D) \to \mathbf{X}_{\mathrm{st}}(D) \to \mathbf{X}_{\mathrm{dR}}((D,\mathrm{Fil})) \to 0.$$
\item For any object $(D, \mathrm{Fil})$ in $\mathbf{MF}^{\mathrm{ad}}_{K,L}(\varphi, N)$, $\mathbf{D}_{\mathrm{st}}(\mathbf{V}_{\mathrm{st}}((D,Fil)))= (D, \mathrm{Fil})$ as submodules of $D\otimes_{K_{0}}\mathbf{B}_{\mathrm{st}}$.
\item For any object $V$ in $\mathbf{Rep}^{\mathrm{st}}_{L}(\mathscr{G}_{K})$, $\mathbf{V}_{\mathrm{st}}(\mathbf{D}_{\mathrm{st}}(V))= V$ as submodules of $V\otimes_{\mathbf{Q}_{p}}\mathbf{B}_{\mathrm{st}}$.
\end{enumerate}
\end{thm}

The second statement of the above theorem is called the fundamental exact sequence. The cohomology groups on $\mathbf{X}_{\mathrm{dR}}$-level are easy to compute. We can use the fundamental exact sequence to give the following simple properties of cohomology groups on $\mathbf{X}_{\mathrm{st}}$-level.

\begin{lem}\label{lem:indepofwt1}
Let $(D,\mathrm{Fil})$ be in $\mathbf{MF}^{\mathrm{ad}}_{K,L}(\varphi, N)$. Suppose that $\mathrm{Fil}^{0}(D_{K}) = D_{K}$. The maps $\Hk{k}{\mathbf{V}_{\mathrm{st}}(D)} \to \Hk{k}{\mathbf{X}_{\mathrm{st}}(D)}$ are isomorphisms for all $k$.
\end{lem}
\begin{proof}
We have  
$$\mathrm{Fil}^{0}(D_{K}\otimes_{K}\mathbf{B}_{\mathrm{dR}}) =\sum_{i\geq 0} \mathrm{Fil}^{i}(D_{K})\otimes_{K}t^{-i}\mathbf{B}_{\mathrm{dR}}^{+}.$$ 
Consider $M_{i} = \sum_{j> i} \mathrm{Fil}^{j}(D_{K})\otimes_{K}t^{-j}\mathbf{B}_{\mathrm{dR}}^{+} + D_{K}\otimes_{K}t^{-i}\mathbf{B}_{\mathrm{dR}}^{+}$ for $i \geq 0$. The module $M_{0}$ is $\mathrm{Fil}^{0}(D_{K}\otimes_{K}\mathbf{B}_{\mathrm{dR}})$. The quotient $M_{i+1}/M_{i}$ is direct sum of $t^{-i-1}\mathbf{B}_{\mathrm{dR}}^{+}/t^{-i}\mathbf{B}_{\mathrm{dR}}^{+}$. The Galois modules $t^{-i-1}\mathbf{B}_{\mathrm{dR}}/t^{-i}\mathbf{B}_{\mathrm{dR}}^{+}$ are cohomologically trivial. Moreover $M_{i}=D_{K}\otimes_{K}t^{-i}\mathbf{B}_{\mathrm{dR}}^{+}$ for $i$ large enough. The quotient $(D_{K}\otimes_{K}\mathbf{B}_{\mathrm{dR}})/M_{i}$ is direct sum of $\mathbf{B}_{\mathrm{dR}}/t^{-i}\mathbf{B}_{\mathrm{dR}}^{+}$. The latter is cohomologically trivial. Hence $\mathbf{X}_{\mathrm{dR}}((D,\mathrm{Fil}))$ is cohomologically trivial. Then the lemma follows from the fundamental exact sequence.
\end{proof}

Let $D$ be a $(\varphi,N)$-module, $\mathrm{Fil}_{1}$,$\mathrm{Fil}_{2}$ two decreasing filtrations on it. We call they are of the same filtration type if for any integer $i$ there exists an integer $j$ such that
$$\mathrm{Fil}^{i}_{1}(D_{K}) = \mathrm{Fil}^{j}_{2}(D_{K}).$$
Note that in this case, the number of Hodge-Tate weights of both filtration are the same.

\begin{lem}\label{lem:indepofwt}
Let $D$ be a $(\varphi,N)$-module, $\mathrm{Fil}_{1}$,$\mathrm{Fil}_{2}$ two admissible filtrations of the same filtration type on it. Let 
$$I_{1} =\{i_{n} < i_{n+1} < \cdots < i_{0} < i_{1} < \cdots < i_{m} \}$$
$$I_{2} =\{j_{n} < j_{n+1} < \cdots < j_{0} < j_{1} < \cdots < j_{m} \}$$
be the set of Hodge-Tate weights of $\mathrm{Fil}_{1}$ and $\mathrm{Fil}_{2}$, respectively. Suppose further that $I_{1}$, $I_{2}$ satisfy the following conditions:
\begin{enumerate}
\item For $n\leq l < 0$, $0> i_{l} \geq j_{l}$.
\item For $0 \leq l \leq m$, $j_{l} \geq i_{l} \geq 0$.
\end{enumerate}
Then the image of $\Hk{k}{\mathbf{V}_{\mathrm{st}}((D,\mathrm{Fil}_{1}))}$ in $\Hk{k}{\mathbf{X}_{\mathrm{st}}(D)}$ is the same as $\Hk{k}{\mathbf{V}_{\mathrm{st}}((D,\mathrm{Fil}_{2}))}$ for all $k$.
\end{lem}
\begin{proof}
By the fundamental exact sequence, we have the image of $\Hk{k}{\mathbf{V}_{\mathrm{st}}((D,\mathrm{Fil}_{i}))}$ in $\Hk{k}{\mathbf{X}_{\mathrm{st}}(D)}$ is the kernel of $\Hk{k}{\mathbf{X}_{\mathrm{st}}(D)}  \to \Hk{k}{\mathbf{X}_{\mathrm{st}}((D,\mathrm{Fil}_{i}))}$ for $i =1 ,2$. Consider the module 
$$M_{0} = \sum_{n \leq l < 0} \mathrm{Fil}^{j_{l}}(D_{K})\otimes_{K}t^{-j_{l}}\mathbf{B}_{\mathrm{dR}}^{+} + \sum_{0 \leq l \leq m} \mathrm{Fil}^{i_{l}}(D_{K})\otimes_{K}t^{-i_{l}}\mathbf{B}_{\mathrm{dR}}^{+}.$$
It is the intersection of $\mathrm{Fil}_{1}^{0}(D_{K}\otimes_{K}\mathbf{B}_{\mathrm{dR}})$ and $\mathrm{Fil}_{2}^{0}(D_{K}\otimes_{K}\mathbf{B}_{\mathrm{dR}})$. Then by the fundamental exact sequence, $\Hk{k}{\mathbf{X}_{\mathrm{st}}(D)}  \to \Hk{k}{\mathbf{X}_{\mathrm{st}}((D,\mathrm{Fil}_{i}))}$ factors through $\Hk{k}{(D_{K}\otimes_{K}\mathbf{B}_{\mathrm{dR}})/M_{0}}$ for $i =1,2$ and all $k$. 

Consider the modules
$$M_{l} = \sum_{n - l \leq a < 0} \mathrm{Fil}^{j_{a}}(D_{K})\otimes_{K}t^{-j_{a}}\mathbf{B}_{\mathrm{dR}}^{+}+\sum_{\substack{0 \leq a \leq m\\ n \leq a < n- l}} \mathrm{Fil}^{i_{a}}(D_{K})\otimes_{K}t^{-i_{a}}\mathbf{B}_{\mathrm{dR}}^{+}$$
for $n\leq l < 0$. The module $M_{n}$ is $\mathrm{Fil}_{1}^{0}(D_{K}\otimes_{K}\mathbf{B}_{\mathrm{dR}})$. The quotient $M_{l}/M_{l+1}$ is direct sum of $t^{-j_{l}}\mathbf{B}_{\mathrm{dR}}^{+}/t^{-i_{l}}\mathbf{B}_{\mathrm{dR}}^{+}$ for  $n\leq l < 0$. The latter one is cohomologically trivial.  So $M_{n}/M_{0}$ is all cohomologically trivial. Then the maps $\Hk{k}{(D_{K}\otimes_{K}\mathbf{B}_{\mathrm{dR}})/M_{0}} \to \Hk{k}{(D_{K}\otimes_{K}\mathbf{B}_{\mathrm{dR}})/M_{n}}$ are isomorphisms for all $k$.

Similarly, consider the modules
$$M_{l} = \sum_{\substack{n  \leq a < 0\\ m-l < a \leq m}} \mathrm{Fil}^{j_{a}}(D_{K})\otimes_{K}t^{-j_{a}}\mathbf{B}_{\mathrm{dR}}^{+}+\sum_{0 \leq a \leq m-l} \mathrm{Fil}^{i_{a}}(D_{K})\otimes_{K}t^{-i_{a}}\mathbf{B}_{\mathrm{dR}}^{+}$$
for $0 < l \leq m+1$. The module $M_{m+1}$ is $\mathrm{Fil}_{2}^{0}(D_{K}\otimes_{K}\mathbf{B}_{\mathrm{dR}})$. We also have that $$\Hk{k}{(D_{K}\otimes_{K}\mathbf{B}_{\mathrm{dR}})/M_{0}} \to \Hk{k}{(D_{K}\otimes_{K}\mathbf{B}_{\mathrm{dR}})/M_{m+1}}$$ 
are isomorphisms for all $k$. So the lemma holds.
\end{proof}

\section{Monodromy modules $(D_{\alpha},\mathrm{Fil}_{m,k,\mathscr{L}})$ of rank two} ~\vskip 1pt

In this section we define  a certain kind of admissible filtered $(\varphi,N)$-modules $(D, \mathrm{Fil})$ of rank $2$. They are called monodromy modules.  

Let $\alpha$ be an element in $L^{\times}$, $m=(m_{\sigma})_{\sigma}$ and $k=(k_{\sigma})_{\sigma}$   elements in $\displaystyle{\oplus_{\sigma}}\mathbf{Z}$ and $\mathscr{L}=(\mathscr{L}_{\sigma})$ an element in $L\otimes_{\mathbf{Q}_{p}}K$. Suppose further that they satisfy the following conditions: 
\begin{itemize}
\item $k_{\sigma} > m_{\sigma}$.
\item $ e\cdot(2v_{p}(\alpha)+f) = \sum_{\sigma}(k_{\sigma}+m_{\sigma})$.
\item $ e\cdot v_{p}(\alpha) \geq \sum_{\sigma}m_{\sigma}$.
\end{itemize}

Define $(D_{\alpha}, \mathrm{Fil}_{m, k,\mathscr{L}})$ as follows:

$$D_{\alpha} := (L\otimes_{\mathbf{Q}_{p}}K_{0})\cdot e_{2} \oplus (L\otimes_{\mathbf{Q}_{p}}K_{0}) \cdot e_{1}$$
$$\begin{array}{cc}
\varphi(e_{2}) = (p\alpha,p,\cdots,p) e_{2}  & \varphi(e_{1}) = (\alpha,1,\cdots,1) e_{1} \\
N(e_{2}) = e_{1} & N(e_{1}) = 0
\end{array}$$
$$\mathrm{Fil}_{m, k,\mathscr{L}}^{i}(D_{\alpha,K}) =\oplus_{\sigma} \mathrm{Fil}_{m_{\sigma}, k_{\sigma},\mathscr{L}_{\sigma}}^{i}((D_{\alpha,K})_{\sigma})$$
$$\mathrm{Fil}_{m_{\sigma},k_{\sigma},\mathscr{L}_{\sigma}}^{i}((D_{\alpha,K})_{\sigma}):=\left\{ \begin{array}{ll}
0 & i > k_{\sigma} \\
(L\otimes_{\mathbf{Q}_{p}}K)_{\sigma} \cdot (e_{2}+\mathscr{L}_{\sigma} e_{1}) & m_{\sigma} < i \leq k_{\sigma} \\
(D_{\alpha,K})_{\sigma} & i \leq m_{\sigma}.
\end{array}
\right.
$$
It is easy to check this module is admissible. We call the filtered $(\varphi,N)$-module as a monodromy $(\varphi,N)$-module of rank $2$. This kind of modules is also considered by B. Schraen. In Proposition 3.1 of \cite{S} he classified all two dimensional semi-stable non crystalline Galois representations with non degenerated weights. Under his terminology, the representations we consider here are just with two dimensional semi-stable non crystalline Galois representations with non degenerated weights and $S= \emptyset$. The parameter $\mathscr{L}$ is called the Fontaine's $\mathscr{L}$-invariant of this module. The Galois representation associated to it is called as a two dimensional Galois representation of monodromy type. The parameter $\mathscr{L}$ is called Fontaine's $\mathscr{L}$-invariant $\mathscr{L}$ of the representation. Up to a twist of a character, we may assume that $m =0 $.

For $\mathscr{L} \neq 0$ and $ e\cdot v_{p}(\alpha)+n \geq \sum_{\mathscr{L}_{\sigma} \neq 0 }m_{\sigma}+\sum_{\mathscr{L}_{\sigma} = 0 }k_{\sigma}$, we can also define an admissible filtered $(\varphi,N)$-module $(D_{\alpha,p\alpha}, \mathrm{Fil}_{m, k,\mathscr{L}})$ by
$$\begin{array}{cc}
\varphi(e_{2}) = (p\alpha,p,\cdots,p) e_{2}  & \varphi(e_{1}) = (\alpha,1,\cdots,1) e_{1} \\
N(e_{2}) = 0 & N(e_{1}) = 0
\end{array}$$
$$\mathrm{Fil}_{m, k,\mathscr{L}}^{i}(D_{\alpha,K}) =\oplus_{\sigma} \mathrm{Fil}_{m_{\sigma}, k_{\sigma},\mathscr{L}_{\sigma}}^{i}((D_{\alpha,K})_{\sigma})$$
$$\mathrm{Fil}_{m_{\sigma},k_{\sigma},\mathscr{L}_{\sigma}}^{i}((D_{\alpha,K})_{\sigma}):=\left\{ \begin{array}{ll}
0 & i > k_{\sigma} \\
(L\otimes_{\mathbf{Q}_{p}}K)_{\sigma} \cdot (e_{2}+\mathscr{L}_{\sigma} e_{1}) & m_{\sigma} < i \leq k_{\sigma} \\
(D_{\alpha,K})_{\sigma} & i \leq m_{\sigma}.
\end{array}
\right.
$$
It is easy to check that two such modules $(D_{\alpha,p\alpha}, \mathrm{Fil}_{m, k,\mathscr{L}})$ and $(D_{\alpha',p\alpha'}, \mathrm{Fil}_{m', k',\mathscr{L}'})$ are isomorphic if and only if $\alpha =\alpha'$, $m =m'$, $k =k'$ and there exists a $b \in L^{\times}$ such that $\mathscr{L}' =b\mathscr{L}'$.

\section{The semi-stable representations $W_{\mathscr{L},k}$} ~\vskip 1pt

\subsection{Definition and basic properties} ~\vskip 1pt

Consider the object $(D, \mathrm{Fil}_{\mathscr{L},k})$ in $\mathbf{MF}^{\mathrm{ad}}_{K,L}(\varphi, N)$ defined as follows where $k=(k_{\sigma})_{\sigma}  \in \oplus_{\sigma}\mathbf{Z}$ and $\mathscr{L} \in L\otimes_{\mathbf{Q}_{p}} K$:
$$D := (L\otimes_{\mathbf{Q}_{p}}K_{0})\cdot f_{1} \oplus (L\otimes_{\mathbf{Q}_{p}}K_{0})\cdot f_{2} \oplus (L\otimes_{\mathbf{Q}_{p}}K_{0}) \cdot f_{3}$$
$$\begin{array}{cccc}
\varphi(f_{1}) = pf_{1},& \varphi(f_{2}) = f_{2} &\mathrm{and}& \varphi(f_{3}) = p^{-1}f_{3} \\
N(f_{1}) = 2f_{2},& N(f_{2}) = f_{3} &\mathrm{and} &N(f_{3}) = 0 \\
g_{\mathscr{L},1}:=f_{1} +2 \mathscr{L} f_{2} + \mathscr{L}^{2} f_{3} ,& g_{\mathscr{L}, 2} := f_{2} + \mathscr{L} f_{3,} &\mathrm{and}& g_{\mathscr{L}, 3} := f_{3}
\end{array}$$
$$\mathrm{Fil}_{\mathscr{L},k}^{i}(D_{K}) =\oplus_{\sigma} \mathrm{Fil}_{\mathscr{L}_{\sigma}, k_{\sigma}}^{i}((D_{K})_{\sigma})$$

$$\mathrm{Fil}_{\mathscr{L}_{\sigma}, k_{\sigma}}^{i}(D_{\sigma}) := \left\{ \begin{array}{ll}
0 & i > k_{\sigma} \\
(L\otimes_{\mathbf{Q}_{p}}K)_{\sigma}\cdot g_{\mathscr{L},1} & 0 < i \leq k_{\sigma} \\
(L\otimes_{\mathbf{Q}_{p}}K)_{\sigma}\cdot g_{\mathscr{L},1} \oplus (L\otimes_{\mathbf{Q}_{p}}K)_{\sigma}\cdot g_{\mathscr{L},2}  & -k_{\sigma}< i \leq 0 \\
(D\otimes_{K_{0}}K)_{\sigma} & i \leq -k_{\sigma}
\end{array}
\right.$$
Denote the semi-stable representation associated to it by $W_{\mathscr{L},k}$.
\begin{prop}\label{prop:end}
Let $(D_{\alpha}, \mathrm{Fil}_{0,k,\mathscr{L}})$ be a monodromy filtered $(\varphi,N)$-module defined in the last section. The representation $W_{\mathscr{L},k}$ is isomorphic to $\mathrm{End}^{0}(\mathbf{V}_{\mathrm{st}}((D_{\alpha}, \mathrm{Fil}_{k,\mathscr{L}})))$.
\end{prop}
\begin{proof}
The dual $(D_{\alpha}^{*}, \mathrm{Fil}_{0,k,\mathscr{L}}^{*})$ of $(D_{\alpha}, \mathrm{Fil}_{0,k,\mathscr{L}})$ is as follows:
$$D_{\alpha}^{*} := (L\otimes_{\mathbf{Q}_{p}}K_{0})\cdot e_{2}^{*} \oplus (L\otimes_{\mathbf{Q}_{p}}K_{0}) \cdot e_{1}^{*}$$
$$\begin{array}{cc}
\varphi(e_{2}^{*}) = (p^{-1}\alpha^{-1},p^{-1},\cdots, p^{-1}) e_{2}^{*}  & \varphi(e_{1}^{*}) = (\alpha^{-1},1,\cdots,1) e_{1}^{*} \\
N(e_{2}^{*}) = 0 & N(e_{1}^{*}) = - e_{2}^{*}
\end{array}$$
$$\mathrm{Fil}_{0,k,\mathscr{L}}^{i}(D^{*}_{\alpha,K}) =\oplus_{\sigma} \mathrm{Fil}^{*,i}_{0, k_{\sigma},\mathscr{L}_{\sigma}}((D^{*}_{\alpha,K})_{\sigma})$$

$$\mathrm{Fil}_{0,k_{\sigma},\mathscr{L}_{\sigma}}^{*,i}(D_{\alpha,\sigma}) := \left\{ \begin{array}{ll}
0 & i > 0 \\
(L\otimes_{\mathbf{Q}_{p}}K) \cdot (e_{1}^{*} - \mathscr{L}_{\sigma} e_{2}^{*}) & -k_{\sigma} < i \leq 0 \\
(D_{\alpha}\otimes_{K_{0}}K)_{\sigma} & i \leq -k_{\sigma}
\end{array}
\right.
$$ where $\{e_{2}^{*}, e_{1}^{*}\}$ is the dual basis of $\{e_{2}, e_{1}\}$. Then mapping $f_{1}$ to $e_{2} \otimes e_{1}^{*}$, $f_{2}$ to $\frac{1}{2}(e_{1} \otimes e_{1}^{*} - e_{2} \otimes e_{2}^{*})$ and $f_{3}$ to $-e_{1}\otimes e_{2}^{*}$ , we get an isomorphism of filtered $(\varphi,N)$-modules between $(D, \mathrm{Fil}_{\mathscr{L},k})$ and $\mathrm{End}^{0}((D_{\alpha}, \mathrm{Fil}_{0,k,\mathscr{L}}))$. So the associated Galois representations are also isomorphic.
\end{proof}

The $(\varphi,N)$ module $D_{\alpha}$ has the natural filtration of $(\varphi,N)$-module $D_{\alpha}^{N=0} \subseteq D_{\alpha}$. It induces the following filtration on $D$:
\begin{gather}\label{eqn:fil}
D_{0}=0\subset D_{1}= D^{N=0} \subset D_{2}=D^{N^{2}=0}  \subset D_{3}=D.
\end{gather}
The free $L\otimes_{\mathbf{Q}_{p}}K_{0}$ modules $D_{i}$($i=1,2,3$) are spanned by $\{f_{j}|1 \leq j \leq i\}$. The quotient modules $D/D_{1}$ and $D/D_{2}$ are $\mathrm{Hom}_{L\otimes_{\mathbf{Q}_{p}}K_{0}}(D_{\alpha}^{N=0},D_{\alpha})$ and $\mathrm{Hom}_{L\otimes_{\mathbf{Q}_{p}}K_{0}}(D_{\alpha}^{N=0},D_{\alpha}/D_{\alpha}^{N=0})$, respectively. So the gradient $D_{2}/D_{1}$ is $\mathrm{End}_{L\otimes_{\mathbf{Q}_{p}}K_{0}}(D_{\alpha}^{N=0})$.

Then we consider the $\mathbf{B}_{e,L}$-module $\mathbf{X}_{\mathrm{st}}(D)$. It has the following filtration
$$0\subset \mathbf{X}_{\mathrm{st}}(D_{1}) \subset \mathbf{X}_{\mathrm{st}}(D_{2})  \subset \mathbf{X}_{\mathrm{st}}(D).$$

We can compute the $\mathbf{B}_{e,L}$-module $\mathbf{X}_{\mathrm{st}}(D)$ as follows.
\begin{prop}
Let
\begin{gather*}
u_{1} = tf_{3}, \quad \quad u_{2}=f_{2}+uf_{3}, 
\intertext{and}
u_{3}=\frac{1}{t}(f_{1} + 2uf_{2} +u^{2}f_{3}).
\end{gather*}
Then $\{u_{j}|1 \leq j \leq i\}$ is a basis of the free $\mathbf{B}_{e,L}$-module $\mathbf{X}_{\mathrm{st}}(D_{i})$ ($i = 1, 2, 3$).
\end{prop}

\begin{proof}
We can check directly $\{u_{j}|1 \leq j \leq i\}$  is contained in $\mathbf{X}_{\mathrm{st}}(D_{i})$ ($i = 1, 2, 3$). On the other hand it is a basis of the free $\mathbf{B}_{\mathrm{st},L}$-module  $D_{i}\otimes_{K_{0}}\mathbf{B}_{\mathrm{st}}$. So it is also a basis of the $\mathbf{B}_{e,L}$-module $\mathbf{X}_{\mathrm{st}}(D_{i})$ .
\end{proof}

As in the $(\varphi,N)$-modules level, the quotient modules $\mathbf{X}_{\mathrm{st}}(D)/\mathbf{X}_{\mathrm{st}}(D_{1})$ and $\mathbf{X}_{\mathrm{st}}(D)/\mathbf{X}_{\mathrm{st}}(D_{2})$ are $\mathrm{Hom}_{\mathbf{B}_{e,L}}(\mathbf{X}_{\mathrm{st}}(D_{\alpha}^{N=0}),\mathbf{X}_{\mathrm{st}}(D_{\alpha}))$ and $\mathrm{Hom}_{\mathbf{B}_{e,L}}(\mathbf{X}_{\mathrm{st}}(D_{\alpha}^{N=0}),\mathbf{X}_{\mathrm{st}}(D_{\alpha}/D_{\alpha}^{N=0}))$, respectively . The gradient $\mathbf{X}_{\mathrm{st}}(D_{2})/\mathbf{X}_{\mathrm{st}}(D_{1})$ is $\mathrm{End}_{\mathbf{}B_{e,L}}(\mathbf{X}_{\mathrm{st}}(D_{\alpha}^{N=0}))$. 

\subsection{Structure of $W_{\mathscr{L},\mathbf{1}}$} ~\vskip 1pt

In this subsection we study the structure of $W_{\mathscr{L},\mathbf{1}}$ in detail where $\mathbf{1} =(1,\cdots,1)$.

\begin{prop}
The filtration (\ref{eqn:fil}) is a filtration of admissible submodules.
\end{prop}
\begin{proof}
We have
$$t_{N}(D_{1})=t_{H}((D_{1}, \mathrm{Fil}_{\mathscr{L},\mathbf{1}})) = t_{N}(D_{2})=t_{H}((D_{2}, \mathrm{Fil}_{\mathscr{L},\mathbf{1}})) = n,$$
so they are all admissible.
\end{proof}

Let $W_{i} = \mathbf{V}_{\mathrm{st}}(D_{i})$. By the above proposition,  the Galois representation $W_{\mathscr{L},\mathbf{1}}$ also has a filtration:
$$W_{0}:=0\subset W_{1} \subset W_{2} \subset W_{3}:=W_{\mathscr{L},\mathbf{1}}.$$
We can give the filtration a more explicit description. Recall that we have the following exact sequence:
$$0\to L(1) \to \mathbf{B}_{\mathrm{cris},L}^{\varphi=p} \to \mathbf{B}_{\mathrm{dR},L}/\mathrm{Fil}^{1}(\mathbf{B}_{\mathrm{dR},L}) \to 0$$
Then there exists an element $v_{\mathscr{L}} \in \mathbf{B}_{\mathrm{cris},L}$ such that $v_{\mathscr{L}} - \mathscr{L} \in L\otimes_{\mathbf{Q}_{p}}\mathrm{Fil}^{1}(\mathbf{B}_{\mathrm{dR},L})$ and this element is uniquely determined up to an element $at$ where $a \in L$. 
\begin{prop}\label{prop:base}
Let
\begin{gather*}
v_{1} = tf_{3}, \quad \quad v_{2}=f_{2}+(u+v_{\mathscr{L}})f_{3}, 
\intertext{and}
v_{3}=\frac{1}{t}(f_{1} + 2(u+v_{\mathscr{L}})f_{2} +(u+v_{\mathscr{L}})^{2}f_{3}).
\end{gather*}
Then $\{v_{j}|1 \leq j \leq i\}$ spans $W_{i}$ ($i = 1, 2, 3$).
\end{prop}
\begin{proof}
One can check directly that $v_{i}$ is in $\mathbf{V}_{\mathrm{st}}((D_{i}, \mathrm{Fil}_{\mathscr{L},k}))$. Take the last one as example.

The $\varphi$-action:
\begin{eqnarray*}
\varphi(v_{3}) &= &\frac{1}{pt}(pf_{1} + 2(pu+pv_{\mathscr{L}})f_{2} +(pu+pv_{\mathscr{L}})^{2}p^{-1}f_{3}) \\
&=&v_{3} 
\end{eqnarray*}
The $N$-action:
\begin{eqnarray*}
N(v_{3}) &= &\frac{1}{t}(N(f_{1}) + 2N((u+v_{\mathscr{L}}))f_{2} +2(u+v_{\mathscr{L}})N(f_{2}) +N((u+v_{\mathscr{L}})^{2})f_{3}) \\
&=&\frac{1}{t}(2f_{2} - 2f_{2} +2(u+v_{\mathscr{L}})f_{3} -2(u+v_{\mathscr{L}})f_{3}) \\
&=& 0
\end{eqnarray*}
The filtration:

\begin{eqnarray*}
v_{3} &= &\frac{1}{t}(f_{1} + 2\mathscr{L} f_{2} + \mathscr{L}^{2} f_{3}) + \frac{2(u+v_{\mathscr{L}}-\mathscr{L})}{t}(f_{2} + \mathscr{L} f_{3})\\ & &+\frac{(v_{\mathscr{L}}-\mathscr{L})^{2}+2u(v_{\mathscr{L}} - \mathscr{L})}{t}f_{3} \\
&\in&\mathrm{Fil}_{\mathscr{L},\mathbf{1}}^{1}(D_{K})\otimes_{K}\mathrm{Fil}^{-1}\mathbf{B}_{\mathrm{dR}}+ \mathrm{Fil}_{\mathscr{L},\mathbf{1}}^{0}(D_{K})\otimes_{K}\mathrm{Fil}^{0}\mathbf{B}_{\mathrm{dR}} +\mathrm{Fil}_{\mathscr{L},\mathbf{1}}^{-1}(D_{K})\otimes_{K}\mathrm{Fil}^{1}\mathbf{B}_{\mathrm{dR}}\\
&=& \mathrm{Fil}^{0}_{\mathscr{L},\mathbf{1}}(\mathbf{X}_{\mathrm{st}}(D))
\end{eqnarray*}
It is obvious that the elements $v_{i}$ are $L$-linear independent. Then the statement follows from comparing the dimensions.
\end{proof}

We identify $W_{i}/W_{i-1}=L\cdot v_{i}$ with $L(2-i)$ from now on(for $i =1, 2, 3$).
\begin{lem}\label{lem:class}
The cohomology class in $H^{1}(L(1))$ associated to the extension $W_{2}$ of $L$ by $L(1)$ is $(\exp(\mathscr{L}))+(p)$.
\end{lem}
\begin{proof} Suppose that $\mathscr{L} = \sum_{i}a_{i}\otimes b_{i}$ such that $b_{i}$ is in $p^{2}\mathscr{O}_{K}$. Consider  the elements $\log [(\sqrt[p^{m}]{\exp(b_{i})})_{m}]$ in $\mathbf{B}_{\mathrm{cris}}^{+}$. They are stable under the $\varphi$-action. We have 
$$\theta(\log [(\sqrt[p^{m}]{\exp(b_{i})})_{m}] - b_{i}) =
\log(\exp(b_{i})) - b_{i} =0.$$ So the element $\log [(\sqrt[p^{m}]{\exp(b_{i})})_{m}] - b_{i}$ is in $\mathrm{Fil}^{1}\mathbf{B}_{\mathrm{dR}}$. Then we can choose 
$$\sum_{i}a_{i}\otimes \log [(\sqrt[p^{m}]{\exp(b_{i})})_{m}] $$
 as $v_{\mathscr{L}}$. For any $g \in \mathscr{G}_{K}$,
\begin{eqnarray*}
(g-1)v_{2} &=& (g-1)( f_{2}+(u+v_{\mathscr{L}})f_{3}) \\
&=&((g-1)u+(g -1)v_{\mathscr{L}})f_{3} \\
&=&(1\otimes\log [(\frac{g(\sqrt[p^{m}]{p})}{\sqrt[p^{m}]{p}})_{m}]+\sum_{i}a_{i}\otimes\log [(\frac{g(\sqrt[p^{m}]{\exp(b_{i})})}{\sqrt[p^{m}]{\exp(b_{i})}})_{m}])f_{3} \\
&=&(1\otimes\log[\epsilon_{m}^{(p)(g)}] + \sum_{i}a_{i}\otimes\log[\epsilon_{m}^{(\exp(b_{i}))(g)}])f_{3} \\
&=&((\exp(\mathscr{L}))+(p))(g)tf_{3} \\
&=& ((\exp(\mathscr{L}))+(p))(g)v_{1}
\end{eqnarray*}
So the cohomology class in $\Hk{1}{L(1)}$ associated to the extension $W_{2}$ is  $(\exp(\mathscr{L}))+(p)$.
\end{proof}

\begin{prop}\label{prop:triancase}
Let $c$ be a cohomolgy class in $\Hk{1}{W_{3}/W_{1}}$. Suppose $c$ vanishes in $\Hk{1}{W_{3}/W_{2}}$ and lies in the image of $\Hk{1}{W_{3}}$. Then it must come from $\Hk{1}{W_{2}/W_{1}}$. Moreover through the identification of $\Hk{1}{W_{2}/W_{1}}$ and $\Hk{1}{L}$ as above, it must be of the form $\frac{1}{n}\mathrm{tr}(\gamma\mathscr{L})\psi_{1}+ \gamma\psi_{2}$ where $\gamma$ is in $L\otimes_{\mathbf{Q}_{p}} K$.
\end{prop}

\begin{proof}
Consider the following exact sequence:
$$0 \to W_{2}/W_{1} \to W_{3}/W_{1} \to W_{3}/W_{2} \to 0.$$
It induces the exact sequence
$$0 \to \Hk{1}{W_{2}/W_{1}} \to \Hk{1}{W_{3}/W_{1}} {\to} \Hk{1}{W_{3}/W_{2}}.$$
Since $c$ vanishes in $\Hk{1}{W_{3}/W_{2}}$ so it must come from $\Hk{1}{W_{2}/W_{1}}$. We have identified $W_{2}/W_{1}$ with $L$ so it can be viewed as a cohomology class $\gamma_{1}\psi_{1} + \gamma_{2}\psi_{2}$ in the subgroup $\Hk{1}{L}$ of $\Hk{1}{W_{3}/W_{1}}$.

Then consider the diagram of exact sequences:
$$0 \to W_{1} \to W_{3} \to W_{3}/W_{1} \to 0.$$
It induces the exact sequence
$$\Hk{1}{W_{3}} \to \Hk{1}{W_{3}/W_{1}} \stackrel{\delta}{\to} \Hk{2}{W_{1}}.$$
Recall that $c$ is in the image of $H^{1}(W_{3})$, so $\delta(\gamma_{1}\psi_{1} + \gamma_{2}\psi_{2})$ is zero.  Then $(\gamma_{1}\psi_{1} + \gamma_{2}\psi_{2})\cup((\exp(\mathscr{L}))+(p))$ is also zero. By Proposition \ref{prop:calcupprod}, we have that $\gamma_{1} =\frac{1}{n}\mathrm{tr} (\gamma_{2}\mathscr{L})$.
\end{proof}

\subsection{A cohomological property} ~\vskip 1pt

In this subsection we establish a cohomological property similar to \ref{prop:triancase} on the level  of $\mathbf{X}_{\mathrm{st}}$.

\begin{prop}\label{prop:cohom}
Let $c$ be a cohomolgy class in $\Hk{1}{\mathbf{X}_{\mathrm{st}}(D)/\mathbf{X}_{\mathrm{st}}(D_{1})}$. Suppose $c$ vanishes in $\Hk{1}{\mathbf{X}_{\mathrm{st}}(D)/\mathbf{X}_{\mathrm{st}}(D_{2})}$ and lies in the image of $\Hk{1}{W_{\mathscr{L},k}}$. Then it must come from $\Hk{1}{W_{2}/W_{1}}$ (identified with $\Hk{1}{L}$) and be of the form $\frac{1}{n}\mathrm{tr} (\gamma\mathscr{L})\psi_{1} + \gamma\psi_{2}$ where $\gamma$ is in $L\otimes_{\mathbf{Q}_{p}} K$.
\end{prop}
\begin{proof}
 By the lemma \ref{lem:indepofwt}, the image of $\Hk{1}{W_{\mathscr{L},k}}$ in $\Hk{1}{\mathbf{X}_{\mathrm{st}}(D)/\mathbf{X}_{\mathrm{st}}(D_{1})}$ is independent of the weight vector $k$. So $c$ lies in the image of $\Hk{1}{W_{\mathscr{L},\mathbf{1}}}$. Consider the commutative diagram
$$\xymatrix
{  W_{3} \ar[r] \ar[d] & W_{3}/W_{1}  \ar[d]\\
 \mathbf{X}_{\mathrm{st}}(D) \ar[r] & \mathbf{X}_{\mathrm{st}}(D)/\mathbf{X}_{\mathrm{st}}(D_{1}). }$$
It induces 
$$\xymatrix
{ \Hk{1}{W_{3}} \ar[r] \ar[d] & \Hk{1}{W_{3}/W_{1}}  \ar[d]\\
 \Hk{1}{\mathbf{X}_{\mathrm{st}}(D)} \ar[r] & \Hk{1}{\mathbf{X}_{\mathrm{st}}(D)/\mathbf{X}_{\mathrm{st}}(D_{1})}. }$$
So $c$ also lies in the image of $\Hk{1}{W_{3}/W_{1}}$.

Then consider the commutative diagram of short exact sequences 
$$\xymatrix
{ 0 \ar[r] & W_{2}/W_{1} \ar[r] \ar[d] & W_{3}/W_{1} \ar[r] \ar[d] & W_{3}/W_{2} \ar[d] \ar[r] &0\\
 0 \ar[r] & \mathbf{X}_{\mathrm{st}}(D_{2})/\mathbf{X}_{\mathrm{st}}(D_{1}) \ar[r] & \mathbf{X}_{\mathrm{st}}(D)/\mathbf{X}_{\mathrm{st}}(D_{1}) \ar[r] & \mathbf{X}_{\mathrm{st}}(D)/\mathbf{X}_{\mathrm{st}}(D_{2}) \ar[r] & 0. }$$
It induces the following commutative diagram exact sequences
$$\xymatrix
{ 0\ar[r] \ar[d] & \Hk{1}{W_{2}/W_{1}} \ar[r] \ar[d] & \Hk{1}{W_{3}/W_{1}} \ar[r] \ar[d] & \Hk{1}{W_{3}/W_{2}} \ar[d] \\
 \Hk{0}{\mathbf{X}_{\mathrm{st}}(D/D_{2})}\ar[r] & \Hk{1}{\mathbf{X}_{\mathrm{st}}(D_{2}/D_{1})} \ar[r] & \Hk{1}{\mathbf{X}_{\mathrm{st}}(D)/\mathbf{X}_{\mathrm{st}}(D_{1})} \ar[r] & \Hk{1}{\mathbf{X}_{\mathrm{st}}(D)/\mathbf{X}_{\mathrm{st}}(D_{2})}. }$$
By Lemma \ref{lem:indepofwt1}, the first, second and fourth columns are all isomorphisms. Then proposition follows from Propostion \ref{prop:triancase}.   
\end{proof}

\section{Extension classes of trivial $(\varphi,N)$-modules}
In this section we consider some extension classes of trivial $(\varphi,N)$-modules. To facilitate the computation, we introduce some notations first. Fixed an embedding $\iota_{0}$ form $K_{0}$ to $L$. Let $\iota_{i}$($ 0 \leq i <f$) be the embedding $\iota_{0}\circ\varphi^{-i}$. Identify $L\otimes_{\mathbf{Q}_{p}}K_{0}$ with $\prod_{i=0}^{f-1}L$ using the map which sends $a\otimes b$ to $(a\cdot\iota_{i}(b))_{i}$. 

Let $\alpha$ be an element in $L$. In this section we consider the $(\varphi,N)$-modules $D'$ of rank $2$ such that under the basis $\{v'_{1}, v'_{2}\}$ the $\varphi$ action is of the form ${\scriptsize{\left ( \!\! \begin{array}{cc} 1 & (\alpha,0,\dots,0) \\ 0  & 1 \end{array}\!\! \right ) }}$ and the $N$-action is $0$. Identifying $L\cdot v_{i}$ with $L$. Then we get an extension class of two trivial $(\varphi,N)$-modules. 

Let $\omega_{0}$ be an element in $\mathbf{Q}_{p}^{\mathrm{ur}} \subset \mathbf{B}_{\mathrm{cris}}$ such that $\varphi^{f}(\omega_{0})-\omega_{0}=1$. Let $\omega_{i}$ be the element $\varphi^{i}(\omega_{0})$ for any $0 \leq i <f$. 
\begin{prop}
The $\mathbf{B}_{e,L}$-module $\mathbf{X}_{\mathrm{st}}(D')$ is free of rank two and spanned by $v'_{1}$ and $v'_{2}+(-\alpha\omega_{i})_{0 \leq i <f}v'_{1}$.
\end{prop}
\begin{proof}
The $\mathbf{B}_{\mathrm{st},L}$-module $D'\otimes_{K_{0}}\mathbf{B}_{\mathrm{st},L}$ is free of rank two and spanned by $v'_{1}$ and $v'_{2}+(-\alpha\omega_{i})_{0 \leq i <f}v'_{1}$. On the other hand,
\begin{align*}
\varphi(v'_{2}+(-\alpha\omega_{i})_{0 \leq i <f}v'_{1})&=v'_{2}+(\alpha,0,\cdots,0)v'_{1}+(-\alpha\varphi^{f}(\omega_{0}),-\alpha\omega_{1},\cdots,-\alpha\omega_{f-1})v'_{1}\\
&=v'_{2}+(-\alpha\omega_{i})_{0 \leq i <f}v'_{1}
\end{align*}
So these two elements are all in $\mathbf{X}_{\mathrm{st}}(D')$. Hence the proposition holds.
\end{proof}
Identify $\mathbf{B}_{e,L}\cdot v'_{i}$($i=1,2$) with $\mathbf{B}_{e,L}$. Then $\mathbf{X}_{\mathrm{st}}(D')$ is  an extension of $\mathbf{B}_{e,L}$ by $\mathbf{B}_{e,L}$.

Consider the filtration $\mathrm{Fil}$ on $D'$ defined by
$$\mathrm{Fil}^{i}(D'_{K}) := \left\{ \begin{array}{ll}
0 & i > 0 \\
D'_{K} & i \leq 0
\end{array}
\right.$$
Under the filtration $D'$ is admissible. The filtrated $(\varphi,N)$-submodule $(L\otimes_{\mathbf{Q}_{p}}K_{0})\cdot v'_{1}$ is a trivial filtered $(\varphi,N)$-module. The quotient $(L\otimes_{\mathbf{Q}_{p}}K_{0})\cdot v'_{2}$ is also trivial. 

Let $V$ be the Galois representation $\mathbf{V}_{\mathrm{st}}((D',\mathrm{Fil}))$.
\begin{prop}
The Galois representation $V$ is spanned by $v'_{1}$ and $v'_{2}+(-\alpha\omega_{i})_{0 \leq i <f}v'_{1}$.
\end{prop}

\begin{proof}
These two elements are $L$-linearly independent and contained in $\mathrm{Fil}^{0}(D'_{\alpha,K}\otimes_{K}\mathbf{B}_{\mathrm{dR}})$. So by the above proposition, $V$ is spanned by them.  
\end{proof}
Identify $L\cdot v'_{i}$($i=1,2)$ with $L$. Then $V$ is an extension of $L$ by $L$. We compute its associated cohomology class in the following proposition.

\begin{prop}
The associated cohomology class of $V$ in $\Hk{1}{L}$ is $-\frac{\alpha}{f} \psi_{1}$.
\end{prop}
\begin{proof}
For any $g$ in $\mathscr{G}_{K}$
\begin{align*}
g(v'_{2}+(-\alpha\omega_{i})_{0 \leq i <f}v'_{1}) -(v'_{2}+(-\alpha\omega_{i})_{0 \leq i <f}v'_{1})&= ((\varphi^{f})^{\frac{\psi_{1}(g)}{f}}(-\alpha\omega_{i})_{0 \leq i <f}-(-\alpha\omega_{i})_{0 \leq i <f})v'_{1}\\
&=-\frac{\psi_{1}(g)}{f}\alpha \cdot v'_{1}
\end{align*}
\end{proof}

\begin{prop}\label{prop:constrain}
The element in $\Hk{1}{\mathbf{B}_{e,L}}$ determined by the following exact sequence 
$$0 \to v'_{1}\mathbf{B}_{e,L} \to \mathbf{X}_{\mathrm{st}}(D') \to v'_{2}\mathbf{B}_{e,L} \to 0.$$ 
comes from $\Hk{1}{L}$. Moreover it is of the form $-\frac{\alpha}{f} \psi_{1}$.
\end{prop}
\begin{proof}
We have the following commutative digram of exact sequences:
$$\xymatrix
{ 0 \ar[r] &  v'_{1}L  \ar[r] \ar[d] & V \ar[r] \ar[d] &  v'_{2}L \ar[d] \ar[r] &0\\
0 \ar[r] &v'_{1}\mathbf{B}_{e,L} \ar[r]&\mathbf{X}_{\mathrm{st}}(D')  \ar[r]&v_{2}'\mathbf{B}_{e,L} \ar[r]& 0.}$$
It induces 
$$\xymatrix
{ \Hk{0}{L}  \ar[r] \ar[d] & \Hk{1}{L} \ar[d] \\
\Hk{0}{\mathbf{B}_{e,L}} \ar[r]&\Hk{1}{\mathbf{B}_{e,L}}.}$$
By Lemma \ref{lem:indepofwt1}, all the columns are isomorphisms. So the proposition follows from the above proposition.
\end{proof}

\section{$\mathscr{L}$-invariants and logarithmic derivatives of eigenvalues of Frobenius} ~\vskip 1pt

In this section we prove Theorem \ref{thm:main}. 
Since the statement is purely local and only related to the derivative of the first order, We only need to consider the case $S=L[x]/(x^{2})$.  

Let $V$ be a free $S$-module of rank $2$ with continuous $\mathscr{G}_{K}$-action. Let $V_{0}$ be $V/xV$ which is a two dimensional continuous $\mathscr{G}_{K}$-representation over $L$. Suppose that $\mathbf{D}_{\mathrm{st}}(V_{0})$ is isomorphic to $(D_{\alpha},\mathrm{Fil}_{\mathscr{L},k})$ for some $\alpha \in L^{\times}$ ,$\mathscr{L} \in L\otimes_{\mathbf{Q}_{p}} K$ and $k \in \oplus_{\sigma}\mathbf{Z}$. Suppose that $(V\otimes_{S}(S\otimes_{\mathbf{Q}_{p}}\mathbf{B}_{\mathrm{cris}})^{\varphi^{f}=\beta})^{\mathscr{G}_{K}}$ is a free $S\otimes_{\mathbf{Q}_{p}}K_{0}$-module of rank $1$ for some some $\beta \in S$.

Denote $\mathrm{End}_{L}(V_{0})$ by $U$. Consider the exact sequence of $\mathscr{G}_{K}$-representations over $L$:
\begin{equation}
0 \to U \to \mathrm{Hom}_{L}(V_{0},V) \to U \to 0 \label{eq:exact}
\end{equation}
 Then we have the connecting map 
$$\delta: U \to \Hk{1}{U}.$$ 
Denote $\mathrm{End}_{L}^{0}(V_{0})$ by $W$. The representation $U$ has the canonical decomposition
$$U = W \oplus L\cdot\mathrm{Id}_{V_{0}}.$$
We identify the trivial representation $L$ with $L\cdot\mathrm{Id}_{V_{0}}$ from now one. Denote the projection of $U$ to $L$ by $\iota$ and the projection of $U$ to $W$ by $\iota_{0}$. Denote the $\mathscr{G}_{K}$-representation $\mathrm{Hom}_{L}(V_{0},V)$ by $\tilde{W}$.

\begin{prop}
The cohomology class $\iota(\delta(\mathrm{Id}_{V_{0}}))$ in $\Hk{1}{L}$ is the continuous homomorphism $\frac{d}{2dx}(\log\circ\det)$ from $\mathscr{G}_{K}$ to $L$.
\end{prop}
\begin{proof}
Fix a basis $\{v_{1}, v_{2}\}$ of $V$. For any $g$ in $\mathscr{G}_{K}$, denote the transformation matrix of $g$ on $V$ by $(1+B_{g}x)A_{g}$ where $A_{g} \in \mathrm{GL}(2, L)$ and $B_{g} \in \mathrm{M}(2, L)$ . Then the transformation matrix of $g$ on $V_{0}$ is just $A_{g}$. The cohomology class $\iota(\delta(\mathrm{Id}_{V_{0}}))$ is just 
$$\frac{1}{2}\mathrm{tr}(((1+B_{g}x)A_{g})A_{g}^{-1}-1)\mathrm{Id}_{V_{0}}=\frac{1}{2}\mathrm{tr}(B_{g})x\mathrm{Id}_{V_{0}}.$$

On the other hand, we have 
\begin{eqnarray*}
\frac{1}{2}\log\circ\det((1+B_{g}x)A_{g}) &=& \frac{1}{2}\log\circ\det((1+B_{g}x)) + \frac{1}{2}\log\circ\det(A_{g}) \\
&=&\mathrm{tr}(B_{g})x + \frac{1}{2}\log\circ\det(A_{g})
\end{eqnarray*}
So the equality $\iota(\delta(\mathrm{Id}_{V_{0}})) = \frac{d}{2dx}(\log\circ\det)$ holds.
\end{proof}
Suppose the logarithm of the determinate of the represention $V$ is of the form $\delta(x)\psi_{1}+\kappa(x)\psi_{2}$. Then $\iota(\delta(\mathrm{Id}_{V_{0}}))=\frac{1}{2}(\delta'\psi_{1}+\kappa'\psi_{2})$ where $\kappa'$(resp. $\delta'$) is the derivative of $\kappa(x)$(resp. $\delta(x)$) at the point zero.

Let $\mathbf{D}_{\mathrm{cris}}(V_{0})$ (resp. $\mathbf{D}_{\mathrm{cris}}(V)$) be the $L\otimes_{\mathbf{Q}_{p}}K_{0}$-module $(V_{0}\otimes_{\mathbf{Q}_{p}}\mathbf{B}_{\mathrm{cris}})^{\mathscr{G}_{K}}$(resp. $(V\otimes_{\mathbf{Q}_{p}}\mathbf{B}_{\mathrm{cris}})^{\mathscr{G}_{K}}$).
\begin{prop}\label{prop:cris}
The $S\otimes_{\mathbf{Q}_{p}}K_{0}$-module $\mathbf{D}_{\mathrm{cris}}(V)$ is $(V\otimes_{S}(S\otimes_{\mathbf{Q}_{p}}\mathbf{B}_{\mathrm{cris}})^{\varphi^{f}=\beta})^{\mathscr{G}_{K}}$ . The modulo $x$ map from $\mathbf{D}_{\mathrm{cris}}(V)$ to $\mathbf{D}_{\mathrm{cris}}(V_{0})$ is surjective.
\end{prop}
\begin{proof}
Consider the exact sequence of $L\otimes_{\mathbf{Q}_{p}}K_{0}$-modules:
$$0 \to V_{0} \to V \to V_{0} \to 0$$
where the map from $V_{0}$ to $V$ is just multiplication by $x$ and the map $V$ to $V_{0}$ is just modulo $x$. It induces the following exact sequence:
$$0 \to \mathbf{D}_{\mathrm{cris}}(V_{0}) \to \mathbf{D}_{\mathrm{cris}}(V) \to \mathbf{D}_{\mathrm{cris}}(V_{0}).$$
The $L\otimes_{\mathbf{Q}_{p}}K_{0}$-module $\mathbf{D}_{\mathrm{cris}}(V_{0})$ is free. It is spanned by $e_{1}$ in $\mathbf{D}_{\mathrm{st}}(V_{0})$. One the other hand the free $S\otimes_{\mathbf{Q}_{p}}K_{0}$-module $(V\otimes_{S}(S\otimes_{\mathbf{Q}_{p}}\mathbf{B}_{\mathrm{cris}})^{\varphi^{f}=\beta})^{\mathscr{G}_{K}}$ of rank $1$ is contained $\mathbf{D}_{\mathrm{cris}}(V)$. Then the proposition follows from comparing the dimensions as $L$-vector spaces.
\end{proof}
\begin{cor}
The element $\beta$ in $S$ is just $\alpha$ modulo $x$.
\end{cor}
\begin{proof}
By the above proposition and its proof, the element $\beta$ modulo $x$ is the eigenvalue of $\varphi^{f}$ on $e_{1}$. 
\end{proof}

\begin{cor}\label{cor:extension}
There exists a lifting $v_{0}$ of the identity map on $\mathrm{End}(\mathbf{D}_{\mathrm{cris}}(V_{0}))$ in the $(\varphi,N)$-module $\mathrm{Hom}(\mathbf{D}_{\mathrm{cris}}(V_{0}), \mathbf{D}_{\mathrm{cris}}(V))$ such that under the basis $\{v_{0},xv_{0}\}$ the $\varphi$-action is given by ${\scriptsize{\left ( \!\! \begin{array}{cc} 1 & (\frac{\beta'}{\alpha},0,\dots,0) \\ 0  & 1 \end{array}\!\! \right ) }}$.
 
\end{cor}
\begin{proof}
Let $v'_{0}$ be a lifting. Then we can take $v_{0}$ as $\sum_{i=0}^{f-1}\varphi^{i}((1,0,\cdots,0)v'_{0})$.
\end{proof}

Let $D$ be the filtered $(\varphi,N)$-module $\mathbf{D}_{\mathrm{st}}(W)$. By our assumption, it is the same as in section 5 so we can define its sub modules $D_{i}$ ($i = 1,2,3$) as in that section. Since $\mathbf{D}_{\mathrm{cris}}(V_{0})$ (resp. $\mathbf{D}_{\mathrm{cris}}(V)$)  is a submodule of $\mathbf{D}_{\mathrm{st}}(V)$ (resp. $\mathbf{D}_{\mathrm{st}}(V)$),  
We have the following $B_{e,L}$-modules:
\begin{align*}
\overline{\mathbf{W}}_{1,e} &=\mathbf{X}_{\mathrm{st}}(\mathrm{Hom}(\mathbf{D}_{\mathrm{cris}}(V_{0}),\mathbf{D}_{\mathrm{st}}(V))) \\
\overline{\mathbf{W}}_{2,e}&=\mathbf{X}_{\mathrm{st}}(\mathrm{Hom}(\mathbf{D}_{\mathrm{cris}}(V_{0}),\mathbf{D}_{\mathrm{st}}(V)/\mathbf{D}_{\mathrm{cris}}(V)))\\
\overline{\mathbf{W}}_{3,e}&=\mathbf{X}_{\mathrm{st}}(\mathrm{Hom}(\mathbf{D}_{\mathrm{cris}}(V_{0}), \mathbf{D}_{\mathrm{cris}}(V))).
\end{align*}

\begin{proof}[Proof of Theorem \ref{thm:main}]
We have the following commutative diagram of exact sequences of Galois modules:
$$\xymatrix
{ 0 \ar[r] & W\oplus L  \ar[r] \ar[d] & \tilde{W} \ar[r] \ar[d] & W\oplus L \ar[d] \ar[r] &0\\
0 \ar[r] &\mathbf{X}_{\mathrm{st}}(D/D_{1}) \ar[r]&\overline{\mathbf{W}}_{1,e}  \ar[r]&\mathbf{X}_{\mathrm{st}}(D/D_{1}) . &}$$
The image of the identity map on $V_{0}$ maps to the identity map on
$\mathbf{D}_{\mathrm{cris}}(V_{0})\otimes_{K_{0}}\mathbf{B}_{\mathrm{st}}$. Let $M_{1}$ be the image of  $\overline{\mathbf{W}}_{1,e}$ in $\mathbf{X}_{\mathrm{st}}(D/D_{1})$. Then the above diagram induces the following commutative diagram:
$$\xymatrix
{ W\oplus L\cdot \mathrm{Id}_{V_{0}} \ar[d] \ar[r] &\Hk{1}{W\oplus L}\ar[d]\\
\Hk{0}{M_{1}}  \ar[r] &\Hk{1}{\mathbf{X}_{\mathrm{st}}(D/D_{1})} . }$$
 The image of $\Hk{1}{L}$ lies in $\Hk{1}{\mathbf{X}_{\mathrm{st}}(D_{2}/D_{1})}$. We compute the cohomology class $\delta(\mathrm{Id}_{V_{0}})$ in $\Hk{1}{\mathbf{X}_{\mathrm{st}}(D/D_{1})}$ by two ways as follows.

On one hand, consider the following diagram of exact sequences of Galois modules:
$$\xymatrix
{ 0 \ar[r] & \mathbf{X}_{\mathrm{st}}(D_{2}/D_{1})  \ar[r] \ar[d] & \overline{\mathbf{W}}_{3,e} \ar[r] \ar[d] & \mathbf{X}_{\mathrm{st}}(D_{2}/D_{1}) \ar[r] \ar[d]&0\\
 0 \ar[r] & \mathbf{X}_{\mathrm{st}}(D/D_{1})   \ar[r] & \overline{\mathbf{W}}_{1,e}\ar[r] & \mathbf{X}_{\mathrm{st}}(D/D_{1}) & }$$
It induces the following commutative diagram:
$$\xymatrix
{ \Hk{0}{\mathbf{X}_{\mathrm{st}}(D_{2}/D_{1})}  \ar[r]^{\delta} \ar[d] &\Hk{1}{\mathbf{X}_{\mathrm{st}}(D_{2}/D_{1})}\ar [d]\\
 \Hk{0}{M_{1}}\ar[r]^{\delta} & \Hk{1}{\mathbf{X}_{\mathrm{st}}(D/D_{1})} . }$$
In $M_{1}$, the image of $\mathrm{Id}_{V_{0}}$ coincides with the image of $\mathrm{Id}_{\mathbf{D}_{\mathrm{cris}}(V_{0})}$. So the cohomology class $\delta(\mathrm{Id}_{V_{0}})$ in $\Hk{1}{\mathbf{X}_{\mathrm{st}}(D_{2}/D_{1})}$ is the same as the image of the following extension class:
$$0 \to  \mathbf{X}_{\mathrm{st}}(D_{2}/D_{1})  \to  \overline{\mathbf{W}}_{3,e} \to   \mathbf{X}_{\mathrm{st}}(D_{2}/D_{1}) \to 0$$
By Proposition \ref{prop:constrain} and Corollary \ref{cor:extension}, the extension class in $\Hk{1}{\mathbf{X}_{\mathrm{st}}(D_{2}/D_{1})}$ comes from $\Hk{1}{L}$ and is of the form $-\frac{\beta'}{f\alpha}\psi_{1}$.

On the other hand, consider the following diagram of exact sequences of Galois modules:
$$\xymatrix
{ 
 0 \ar[r] & \mathbf{X}_{\mathrm{st}}(D/D_{1})   \ar[r] \ar[d]& \overline{\mathbf{W}}_{1,e}\ar[r] \ar[d]& \mathbf{X}_{\mathrm{st}}(D/D_{1}) \ar[d] & \\
0 \ar[r] &\mathbf{X}_{\mathrm{st}}(D/D_{2}) \ar[r] &\overline{\mathbf{W}}_{2,e}  \ar[r]&\mathbf{X}_{\mathrm{st}}(D/D_{2})  &}$$
Let $M_{2}$ be the image of $\overline{\mathbf{W}}_{2,e}$ in $\mathbf{X}_{\mathrm{st}}(D/D_{2})$. Combined with the diagram in the above paragraph, it induces
$$\xymatrix
{ \Hk{0}{\mathbf{X}_{\mathrm{st}}(D_{2}/D_{1})}  \ar[r]^{\delta} \ar[d] &\Hk{1}{\mathbf{X}_{\mathrm{st}}(D_{2}/D_{1})}\ar [d]\\
 \Hk{0}{M_{1}}\ar[r]^{\delta} \ar[d] & \Hk{1}{\mathbf{X}_{\mathrm{st}}(D/D_{1})}\ar[d]\\
  \Hk{0}{M_{2}}\ar[r]^{\delta}  & \Hk{1}{\mathbf{X}_{\mathrm{st}}(D/D_{2})}. }$$
Since both the image of $\delta(\mathrm{Id}_{V_{0}})$ and  $\iota(\delta(\mathrm{Id}_{V_{0}}))$ in $\Hk{1}{\mathbf{X}_{\mathrm{st}}(D/D_{1})}$ lie in $\Hk{1}{\mathbf{X}_{\mathrm{st}}(D_{2}/D_{1})}$, they vanishes in $\Hk{1}{\mathbf{X}_{\mathrm{st}}(D/D_{2})}$. Then so does the image of $\iota_{0}(\delta(\mathrm{Id}_{V_{0}}))$ in $\Hk{1}{\mathbf{X}_{\mathrm{st}}(D/D_{1})}$. By Proposition \ref{prop:cohom} it comes from $\Hk{1}{L}$ and is of the form $\frac{1}{n}\mathrm{tr} (\gamma\mathscr{L})\psi_{1} + \gamma\psi_{2}$ where $\gamma$ is in $L\otimes_{\mathbf{Q}_{p}} K$. Hence the cohomology class $\delta(\mathrm{Id}_{V_{0}})$  in $\Hk{1}{\mathbf{X}_{\mathrm{st}}(D_{2}/D_{1})}$ also equals to
$$\frac{1}{2}(\delta'\psi_{1}+\kappa'\psi_{2})+ \frac{1}{n}\mathrm{tr} (\gamma\mathscr{L})\psi_{1} + \gamma\psi_{2}.$$

Combining the above two paragraphs, we get
$$\frac{1}{2}(\delta'\psi_{1}+\kappa'\psi_{2})+ \frac{1}{n}\mathrm{tr} (\gamma\mathscr{L})\psi_{1} + \gamma\psi_{2}=-\frac{\beta'}{f\alpha}\psi_{1}.$$
The theorem then follows from the identity.
\end{proof}

Similarly, we have a ``degenerated'' formula for the filtered $(\varphi,N)$-module defined in Section 4.
 
\begin{prop}\label{prop:degen}
Let $\mathbf{V}$ be a family of two dimensional Galois representations over $\mathscr{X}$. Let $\log \det \mathbf{V} = \delta \psi_{1}+ \kappa \psi_{2}$  where $\delta \in S$ and $\kappa \in K\otimes_{\mathbf{Q}_{p}}S$. Suppose that $((\mathbf{V}\otimes_{\mathbf{Q}_{p}}\mathbf{B}_{\mathrm{cris}})^{\varphi^{f}=\alpha})^{\mathscr{G}_{K}}$ is a free $S\otimes_{\mathbf{Q}_{p}}K_{0}$-module of rank $1$. Let $x$ be a point in $\mathscr{X}(L)$ such that $V_{x}$ is semi-stable and its associated filtered $(\varphi,N)$-module is isomorphic to $(D_{\alpha,p\alpha}, \mathrm{Fil}_{m, k,\mathscr{L}})$ defined in Section 4. Then the differential forms
$$\mathrm{tr} (\mathscr{L}\cdot d\kappa)$$
vanishes at the point $x$.
\end{prop}
\begin{proof}
Up to a twist of characters, we can assume that $m= 0$. In this case, Proposition \ref{prop:end} holds for the same filtered $\varphi$-module with $N = 0$. We have similar bases as in Proposition \ref{prop:base} without the $u$-term. Similar to Lemma \ref{lem:class}, in this case the corresponding cohomology class is $(\exp(\mathscr{L}))$. The cohomology property we need as in Proposition \ref{prop:cohom} is that in this case the class $c=a\psi_{1}+\gamma\psi_{2}$ must satisfy $\mathrm{tr} (\gamma\mathscr{L})=0$.     Then the proposition follows from the same calculation in Theorem \ref{thm:main}. 
\end{proof}

\end{document}